\documentclass[11 pt, reqno]{amsart}

\usepackage{latexsym}
\usepackage{amssymb}
\usepackage{mathrsfs}
\usepackage{amsmath}
\usepackage{fancybox,color}
\usepackage{enumerate}
\usepackage[latin1]{inputenc}

\usepackage[colorlinks]{hyperref}
\usepackage{eurosym}
\usepackage{url}
\usepackage{graphicx} 
\newcommand{\kom}[1]{}
\renewcommand{\kom}[1]{{\bf [#1]}}

 \def\1{\raisebox{2pt}{\rm{$\chi$}}}
\def\a{{\bf a}}

\newtheorem{theorem}{Theorem}[section]

\newtheorem{lemma}[theorem]{Lemma}

\newtheorem{remark}[theorem]{Remark}
\newtheorem{example}[theorem]{Example}

\newcommand{\R}{{\mathbb R}}

 \newcommand{\eps}{{\varepsilon}}
 \def\1{\raisebox{2pt}{\rm{$\chi$}}}
 
\newcommand{\abs}[1]{\left|#1\right|}
\newcommand{\norm}[1]{\left|\left|#1\right|\right|}
\newcommand{\Rn}{\mathbb{R}^n}

\def\vint_#1{\mathchoice%
          {\mathop{\kern 0.2em\vrule width 0.6em height 0.69678ex depth -0.58065ex
                  \kern -0.8em \intop}\nolimits_{\kern -0.4em#1}}%
          {\mathop{\kern 0.1em\vrule width 0.5em height 0.69678ex depth -0.60387ex
                  \kern -0.6em \intop}\nolimits_{#1}}%
          {\mathop{\kern 0.1em\vrule width 0.5em height 0.69678ex
              depth -0.60387ex
                  \kern -0.6em \intop}\nolimits_{#1}}%
          {\mathop{\kern 0.1em\vrule width 0.5em height 0.69678ex depth -0.60387ex
                  \kern -0.6em \intop}\nolimits_{#1}}}
\def\vintslides_#1{\mathchoice%
          {\mathop{\kern 0.1em\vrule width 0.5em height 0.697ex depth -0.581ex
                  \kern -0.6em \intop}\nolimits_{\kern -0.4em#1}}%
          {\mathop{\kern 0.1em\vrule width 0.3em height 0.697ex depth -0.604ex
                  \kern -0.4em \intop}\nolimits_{#1}}%
          {\mathop{\kern 0.1em\vrule width 0.3em height 0.697ex depth -0.604ex
                  \kern -0.4em \intop}\nolimits_{#1}}%
          {\mathop{\kern 0.1em\vrule width 0.3em height 0.697ex depth -0.604ex
                  \kern -0.4em \intop}\nolimits_{#1}}}

\newcommand{\kint}{\vint}

\newcommand{\aveint}[2]{\mathchoice%
          {\mathop{\kern 0.2em\vrule width 0.6em height 0.69678ex depth -0.58065ex
                  \kern -0.8em \intop}\nolimits_{\kern -0.45em#1}^{#2}}%
          {\mathop{\kern 0.1em\vrule width 0.5em height 0.69678ex depth -0.60387ex
                  \kern -0.6em \intop}\nolimits_{#1}^{#2}}%
          {\mathop{\kern 0.1em\vrule width 0.5em height 0.69678ex depth -0.60387ex
                  \kern -0.6em \intop}\nolimits_{#1}^{#2}}%
          {\mathop{\kern 0.1em\vrule width 0.5em height 0.69678ex depth -0.60387ex
                  \kern -0.6em \intop}\nolimits_{#1}^{#2}}}
\newcommand{\ud}{\, d}
\newcommand{\half}{{\frac{1}{2}}}

\newcommand{\ol}{\overline}
\newcommand{\Om}{\Omega}
\newcommand{\I}{\textrm{I}}
\newcommand{\II}{\textrm{II}}

\newcommand{\tr}{\operatorname{tr}}
\renewcommand{\a}{\alpha}

\makeatletter\def\SL@eqnlefttext #1{\hbox to 0pt{\kern 75 pt %or something else
\llap{\SL@margintext{#1}\quad}\hss}}
\makeatother

\begin{document}

\title[Coupling and Ishii-Lions]{Connections between coupling and Ishii-Lions methods for tug-of-war with noise stochastic games}
\author[Anttila]{Riku Anttila}
\address{Department of Mathematics and Statistics, University of
Jyv\"askyl\"a, PO~Box~35, 40014 Jyv\"askyl\"a, Finland}
\email{riku.t.anttila@jyu.fi}

\author[Manfredi]{Juan J.~Manfredi}
\address{Department of Mathematics, University of
Pittsburgh, PA 15260, USA}
\email{manfredi@pitt.edu}

\author[Parviainen]{Mikko Parviainen}
\address{Department of Mathematics and Statistics, University of
Jyv\"askyl\"a, PO~Box~35, 40014 Jyv\"askyl\"a, Finland}
\email{mikko.j.parviainen@jyu.fi}

\date{\today}
\thanks{MP is supported by the Research Council of Finland, project 360185. RA was supported by Eemil Aaltonen foundation. Part of this research was done during a visit of JJM to the University of Jyv\"askyl\"a in 2024 under the JYU Visiting Fellow Programme. JJM is supported by the Simons 
Collaboration Gift 962828}

\begin{abstract}
We present  a streamlined account of two different regularity methods as well as their connections. 
We consider the coupling method in the context of tug-of-war with noise stochastic games, and consider viscosity solutions of the $p$-Laplace equation in the context of the Ishii-Lions method. 
\end{abstract}

\maketitle

\tableofcontents

\section{Introduction}

In this paper, we carefully establish the connection between the coupling method in the theory of stochastic games and the Ishii-Lions regularity method in the theory of viscosity solutions. The coupling method has its origins in the 1986 work of Lindvall and Rogers \cite{lindvallr86}. In the theory of viscosity solutions, the coupling method is connected to the doubling of variables technique and the theorem on sums within the Ishii-Lions PDE regularity framework introduced in \cite{ishiil90}. These developments appear to have initially evolved as independent threads.

In Section~\ref{sec:background}, we will review the random walk, tug-of-war and tug-of-war with noise and related dynamic programming principles (DPPs). Moreover, we will explain the main lines of the method of couplings starting from the simpler cases and proceeding to the tug-of-war with noise, for which the method of couplings was established in \cite{luirop18}.  The idea is that looking at the games starting at two different points, this may be interpreted as a single game in a product space where the dimension is doubled by introducing suitable couplings for the original probability measures. This also gives rise to $2n$-dimensional dynamic programming principles ($2n$-DPPs). 

In terms of the tug-of-war game, it was discovered in \cite{peresssw09} while the connection between tug-of-war with noise and the $p$-Laplacian was observed in \cite{peress08}. In this work, we use the version developed in \cite{manfredipr12}. It would be almost impossible to cover all the recent tug-of-war literature,  but we mention some additional references later on.

We will also explain how the actions needed to prove regularity estimates are motivated by stochastic intuition. At first sight, the Ishii-Lions method,  in particular the choices made in the theorem on sums,  might appear less intuitive. Therefore, we start with the method of couplings and $2n$-DPPs and derive the corresponding $2n$-dimensional partial differential equation ($2n$-PDE).  Next, we establish the connection between $2n$-PDE and the Ishii-Lions method.  This is done in Sections~\ref{sec:laplacian}, \ref{sec:infty} and \ref{sec:p-lap} for the Laplacian and random walk, the infinite Laplacian and tug-of-war game, as well as the $p$-Laplacian and the tug-of-war with noise, respectively. 

For example in terms of the random walk ($p=2$) the DPP for the value of the process reads as 
\begin{align*}
u_{\eps}(x)=\kint_{B_{\eps}(x)}u_{\eps}\ud z:=\frac{1}{\abs{B_{\eps}}}\int_{B_{\eps}(x)}u_{\eps}\ud z
\end{align*}
and the $2n$-DPP as
\begin{align*}
G(x,y)&=\kint_{B_{\eps}(0)}G(x+h, y+P_{x,y}(h)) \ud h,    
\end{align*}
where $P_{x,y}$ is a reflection coupling; the stochastic heuristics for this choice and the rest of the coupling method, as mentioned, will be explained later. This gives rise to $2n$-PDE
\begin{align*}
\tr\left\{ \begin{pmatrix}
I&B\\
B^T&I 
\end{pmatrix} D^2 G(x,y)\right\}=0,
\end{align*}
through the Taylor theorem with a certain explicit matrix $B$. On the other hand, starting from a harmonic function $u$ and setting $G(x,y)=u(x)-u(y)$, we obtain a solution to the above PDE. Next, we focus our attention on proving the comparison principles for this $2n$-PDE and on finding  a suitable supersolution that implies the desired regularity through comparison: this is actually the Ishii-Lions regularity method for PDEs.

In the continuous-time case and for uniformly parabolic PDEs, the connection between these different regularity methods is known to experts. For example in \cite[Appendix A]{porrettap13} the authors describe  the connection between the coupling method for continuous time stochastic processes and the Ishii-Lions method, but for different PDEs and processes from the ones covered in this paper.  Here we work in the nonlinear context with the coupling method for the tug-of-war games with noise and the Ishii-Lions regularity method for the $p$-Laplace equation. 
Stochastic games are easier to understand in discrete time, where one can immediately describe concepts such as taking turns. In this paper, we work in discrete time and try to give an accessible exposition of the topic starting with simple cases and dropping the excess terms when explaining the ideas.

For a discussion and further references on the coupling method, one can consult \cite{kusuoka15} and \cite{kusuoka17}. The Ishii-Lions method has been used for $p$-Laplace type equations in \cite{attouchipr17} and \cite{siltakoski21}.

 %%%%%%%%%%%%%%%%%%
\section{Background}
\label{sec:background}

\subsection{Tug-of-war games} In terms of games, this section takes inspiration in part from \cite{parviainen24}. More details can also be found in \cite{peresssw09}, \cite{peress08}, \cite{luirops14}, \cite{manfredipr12}, \cite{blancr19b}, and \cite{lewicka20}. As for PDEs, the reader could consult \cite{lindqvist17} or \cite{heinonenkm06} for example.

\noindent{\bf Random walk on balls:}  
Let $\Omega \subset \mathbb{R}^n$ be a bounded domain. Consider a random walk starting from an initial point $x_0 \in \Omega$. The next point 
$$x_1 \in B_\varepsilon(x_0)$$ 
is randomly selected using the uniform distribution over the open ball $B_\varepsilon(x_0)$, which has  radius $\varepsilon$ centered at $x_0$. This process is repeated, with each step determined in a similar way, until the walk exits the domain $\Omega$. The first point outside the domain is denoted by $x_\tau$, and at this point we assign a payoff $F(x_\tau)$. Then,  the value of the process is denoted by 
\[
\begin{split}
u_{\eps}(x_0):=\mathbb E^{x_0}[F(x_{\tau})],
\end{split}
\]  
where  $E^{x_0}$ denotes the expectation when starting at $x_0$. The value satisfies the dynamic programming principle (DPP)
\begin{align*}
u_{\eps}(x)=\kint_{B_{\eps}(x)}u_{\eps}\ud z:=\frac{1}{\abs{B_{\eps}}}\int_{B_{\eps}(x)}u_{\eps}\ud z.
\end{align*}
This formula can be understood by considering one step and adding/integrating possible outcomes against the corresponding probability measure.
For more details, see, for example, \cite{manfredipr12, luirops14}.

\noindent{\bf Tug-of-war:}  
Next, we introduce two competing players, Player I (PI) and Player II (PII). A token is initially placed at $x_0 \in \Omega \subset \mathbb{R}^n$, a fair coin is tossed and the winner of the toss moves the token to a new point $x_1 \in B_\varepsilon(x_0)$. The game proceeds similarly until the token exits the domain $\Omega$. At the end of the game, PII pays PI an amount determined by the payoff function $F(x_\tau)$, where $x_\tau$ is the first point outside $\Omega$. Thus, PII is the minimizing player and PI the maximizing player. The value of the game is 
\begin{align}
\label{eq:elliptic-value-infinity}
u_\eps(x_0):=\sup_{S_{\I}}\inf_{S_{\II}}\,\mathbb{E}_{S_{\I},S_{\II}}^{x_0}[F(x_\tau)],
\end{align}
as long as one suitably penalizes the sequences that do not end the game.
It can also be shown that  the value satisfies the DPP  \cite{armstrongs12,lius15}
\begin{align*}
u_\varepsilon(x)=\frac{1}{2}\sup_{B_\varepsilon(x)}u_\varepsilon+\frac{1}{2}\inf_{B_\varepsilon(x)}u_\varepsilon.
\end{align*}
This formula can be understood by considering one step and summing up the possible outcomes (either PI or PII wins) with the corresponding probabilities and it 
is related to the game theoretic or normalized infinity Laplace equation
\begin{align*}
0=\Delta^{N}_{\infty}u:=\abs{Du}^{-2}\sum_{i,j=1}^{n} u_{x_ix_j} u_{x_i}u_{x_j}=\abs{Du}^{-2} \langle D^2 uDu,Du \rangle,
\end{align*}
where we assumed where we assumed $Du\neq 0$. Indeed, when $u$ is $C^2$ and $Du\neq 0$, we have 
$$\frac12\left\{
\sup_{z \in B_\eps(x)} u(z) +\inf_{z \in B_\eps(x)} u(z) 
\right\}= u(x)+ \frac{\eps^2}{2} \Delta_\infty^N u(x) + o(\eps^2).
$$
See Theorem 3.5 in \cite{lewicka20}.\par
\noindent{\bf Tug-of-war with noise:}  
 This combines the random walk and the tug-of-war game. A token is initially placed at $x_0 \in \Omega \subset \mathbb{R}^n$, and a biased coin is tossed with probabilities $\alpha$ and $\beta$ (where $\alpha + \beta = 1$, $\beta > 0$, $\alpha \geq 0$). If heads (with probability $\beta$) occur, the next point $x_1 \in B_\varepsilon(x_0)$ is randomly chosen as in the random walk on balls, following the uniform distribution over $B_\varepsilon(x_0)$. 

On the other hand, if tails (with probability $\alpha$) occur, a fair coin is thrown and the winner of the toss moves the token to a new point $x_1 \in B_\varepsilon(x_0)$. The game proceeds similarly until the token exits the domain $\Omega$. At the end of the game, Player II pays Player I an amount determined by the pay-off function $F(x_\tau)$, where $x_\tau$ is the first point outside $\Omega$. Then we define
\begin{align}
\label{eq:elliptic-value}
u_\eps(x_0):=\sup_{S_{\I}}\inf_{S_{\II}}\,\mathbb{E}_{S_{\I},S_{\II}}^{x_0}[F(x_\tau)],
\end{align}
where  $x_0$ is a starting point of the game and $S_{\I},S_{\II}$ are the strategies of the players. By \cite{luirops14}, it is known that the order of $\inf$ and $\sup$ is irrelevant, and thus we say that the game has a value. Also, observe that the game ends almost surely in this case because of noise.  The value satisfies by \cite{luirops14}, the DPP
\begin{equation}
\label{eq:DPP-repeat}
\begin{split}
u_{\eps}(x) =&\frac{\a}{2}\left\{ \sup_{B_{\eps}(x)} u_{\eps} + \inf_{B_{\eps}(x)}u_{\eps}\right\}+ \beta
\kint_{B_{\eps}(0)} u_{\eps}(x+h) \ud h,
\end{split}
\end{equation}  
which can be heuristically understood as considering one step and summing up the possible outcomes with the corresponding probabilities. 

 It is well-known by the references at the beginning of the section that this game approximates viscosity solutions to the normalized or game theoretic $p$-Laplace equation 
\begin{align*}
0=\Delta_{p}^{N} u=\Delta u+(p-2)\Delta_{\infty}^{N} u
\end{align*} 
where we again assume  $Du\neq 0$,
\begin{align*}
    \alpha=\frac{p-2}{p+n},\qquad \beta=\frac{n+2}{p+n}.
\end{align*}
and $p>2$

\subsection{Couplings}
\label{sec:intuition}

The regularity argument sketched in this section has connections to the method of couplings of stochastic processes dating back to the 1986 paper of Lindvall and Rogers \cite{lindvallr86}. In the linear case it is well known in the context of PDEs, see, for example, \cite{porrettap13, kusuoka15}.

\noindent {\bf Random walk:} 
We consider a random walk and observe that we can write the difference as an average value in double variables with respect to many measures obtained by coupling
\begin{align}
\label{eq:2nDPP-p-2}
G(x,y)&:=u_{\eps}(x)-u_{\eps}(y)=\kint_{B_{\eps}(0)}u_{\eps}(x+h) \ud h-\kint_{B_{\eps}(0)}u_{\eps}(y+h) \ud h\nonumber\\
 &=\kint_{B_{\eps}(0)}u_{\eps}(x+h)- u_{\eps}(y+P_{x,y}(h)) \ud h\nonumber\\
 &=\kint_{B_{\eps}(0)}G(x+h, y+P_{x,y}(h)) \ud h, 
\end{align}
where $P_{x,y}$ is any isometry of $B_\eps(0)$. Proceeding 
in this way,  the question about the regularity of $u_{\eps}$
is converted into a question about the \textit{size} of a solution $G$ of \eqref{eq:2nDPP-p-2} in $\Omega\times\Omega\subset \R^{2n}$. 
The DPP in $2n$-dimensions also inherits  the monotonicity property of the $n$-dimensional DPP; the inequality  $G_1(x,y)\ge G_2(x,y)$ implies $$
\kint_{B_{\eps}(0)}G_1(x+h, z+P_{x,y}(h)) \ud h\ge \kint_{B_{\eps}(0)}G_2(x+h, z+P_{x,y}(h)) \ud h.$$

We further have a stochastic interpretation considering a process in $\R^{2n}$ defined as follows: When the process is at $(x,y)\in \Om\times \Om$,  the next point is chosen according to the higher-dimensional probability measure that can be read from the last line of \eqref{eq:2nDPP-p-2}. This measure was obtained by coupling the original probability measures (the uniform measures in $B_{\eps}(x)$ and $B_{\eps}(y)$) through the choice of the isometry $P_{x,y}$. This further allows us to define a value for the $2n$-dimensional process analogously to the $n$-dimensional case for the use of this heuristic discussion. 
Next,  we explain the idea of getting a bound for $|G(x,y)|$ using stochastic intuition in $\R^{2n}$. Observe that $G=0$ in the diagonal set $$ T = \{(x, y): x = y\}.$$
Thus we define the following stopping rules and new payoffs for the same process to define another value:
\begin{enumerate}
\item we stop if we reach $T$, and payoff is zero
\item we stop of either $x$ or $y$ goes outside $B_1\subset \Omega$, and  payoff is  $2\sup |u_{\eps}|$
\end{enumerate}
The value of this process with the above boundary values should evidently be greater than or equal to $G$ since the boundary values defined above are greater than those taken by $G$. 
 
We want to get an upper bound for the other value defined above, and thus for the original value. The favorable case is to reach $x=y$ before $x$ or $y$ goes outside $B_1$, because we have boundary values $0$ there. As we try to reach the target, a good choice for $P_{x,y}(h)$ is the reflection of $h$ with respect to $V^\perp=\operatorname{span}(x-y)^\perp$,  when aiming at $x=y$ in the course of the process. In contrast, choosing, for example, the same vector $h$ instead of the reflection would be a bad choice with this aim.

The alternative analytic perspective involves finding a suitable super-solution $f$ to the multidimensional DPP \eqref{eq:2nDPP-p-2} that provides the required bound for the solution $G$, once a comparison principle has been established. 
We now outline the key steps of this approach; the details can be found in \cite[Sections 3 and 4]{luirop18}.
Consider the test function
\begin{align*}
f(x,y) = C\abs{x-y}^{\delta}.
\end{align*}
We want to show that $G\le f$ in $B_1\times B_1$. Let us assume that
the maximum of the difference $G-f$ is reached at a point $(x_0, y_0)\in B_1\times B_1$. 
It turns out that one can always choose $C$ and modify $f$  so that this can be accomplished. 
Suppose that
\begin{equation}\label{thesis}
    \max_{B_1\times B_1}(G(x,y)-f(x,y))=G(x_0,y_0)-f(x_0,y_0)=\theta >0.
\end{equation}

We conclude that the function $\theta+f(x,y)$
touches $G(x,y)$ from above at the point $(x_0, y_0)$.
We have
\begin{align}
\label{eq:bound-cp}
\left\{ \begin{array}{rcl}
f(x,y)+\theta &\ge&  G(x,y) \text{  in  }B_1\times B_1\\
f(x_0,y_0)+\theta &= & G(x_0,y_0).
\end{array}\right.     
\end{align}
We now apply the monotonicity of the $2n$-DPP to obtain 
\begin{align}
\label{eq:DPP-from-cp}
\begin{array}{rcl}
   \theta & = &  G(x_0,y_0) -f(x_0,y_0)   \\
     & = & \kint_{B_{\eps}(0)}G(x_0+h, y_0+P_{x_0,y_0}(h)) \ud h- f(x_0, y_0)\\
     & \le & \kint_{B_{\eps}(0)}f(x_0+h, y_0+P_{x_0,y_0}(h))  \ud h +\theta-f(x_0,y_0).
\end{array}    
\end{align}
On the other hand, we can prove the following lemma using the precise form of $f$. The lemma contradicts our initial hypothesis \eqref{thesis}. We conclude that $\theta=0$ is so that
$u_\eps(x)-u_\eps(y)=g(x,y)\le f(x,y)= C|x-y|^\delta$.
It turns out that using reflection coupling is a beneficial strategy even in the analytic approach when proving this lemma. 
\begin{lemma}\label{testp2}
Let $x_0, y_0$, $P_{x_0,y_0}$ and $f$ be as above. Then
   \begin{align}
 \label{eq:int-ave-ineq}
\kint_{B_{\eps}(0)}f(x_0+h, y_0+P_{x_0,y_0}(h)) \ud h-f(x_0,y_0)<0.
\end{align}
\end{lemma}
%%%%%%%%%%%%%%%%
%
\begin{proof}
We use Taylor's expansion for $f(x,y)=C\abs{x-y}^{\delta}$
\begin{equation}
\label{eq:taylor}
\begin{split}
f&(x+h_x,y+h_y)\\
&=f(x,y)+C\delta|x-y|^{\delta-1}(h_x-h_y)_V\\
&+\frac{C}{2}\delta|x-y|^{\delta-2}\big((\delta-1)(h_x-h_y)^2_V+(h_x-h_y)^2_{V^{\perp}}\big)+\text{error},
\end{split}
\end{equation}
where $+\text{error}$ denotes the usual error term in the Taylor theorem,  $V$ is the space spanned by $x-y$, $(h_x-h_y)_V$  refers to the scalar projection onto $V$ i.e.\ $\langle(h_x-h_y), (x-y)/\abs{x-y}\rangle$, and $(h_x-h_y)_{V^{\perp}}$ onto the orthogonal complement. 
Observe that the above computation is only possible when $f$ is regular enough, that is, $x_0\neq y_0$. This follows from the counterassumption (\ref{thesis}): this is exactly where we use $\theta>0$.

When computing the integral in (\ref{eq:int-ave-ineq}) of $f$, the first-order term produces zero by symmetry. Looking at the second-order terms, we observe that the desired negative output is obtained by the term
$$(\delta-1)(h_x-h_y)^2_V.$$ However, the term $$(h_x-h_y)^2_{V^{\perp}}$$ gives an undesired positive effect. Thus, using reflection gives the desired negative output, and the undesired term vanishes.
\end{proof}

We remark that we have also omitted discretization errors, which would add an additional correction function to $f$ as detailed in \cite[(2.5)]{luirop18}. It also follows that $\theta>0$ is not a sufficient counter-assessment if these errors are taken into account. The actual assumption is \cite[(3.12)]{luirop18}, and there is an explanation on where it is needed right before  estimate (3.15) of the same paper: the point is that in (\ref{eq:DPP-from-cp}) the points might be taken slightly outside the domain, where the functions might jump. Thus, to use (\ref{eq:bound-cp}) the value of $\theta$ also needs to be slightly larger. 
% %%%%%%%

\noindent {\bf Tug-of-war:} 
The following argument is given in detail in Sections 2.3 and 3.1 in \cite{luirop18}. Later this approach was extended to the parabolic setting in \cite{parviainenr16} (despite the difference in publication years).   
Again
 $$
 G(x,y):=u_{\eps}(x)-u_{\eps}(y)
 $$ can be written as a solution of a certain monotone  DPP in $\R^{2n}$:
 \begin{align}\label{eq:DPP-2}
		G&(x,y)=u_\eps(x)-u_\eps(y)\nonumber\\
		&=\frac{1}{2}\,\Bigg\{\,\sup_{\tilde x\in B_\eps(x)}u_\eps(\tilde x)+\inf_{\tilde x\in B_\eps(x)}u_{\eps}(\tilde x)-\sup_{\tilde y\in B_\eps(y)}u_\eps(\tilde y)-\inf_{\tilde y\in B_\eps(y)}u_{\eps}(\tilde y)\,\Bigg\}\nonumber \\
		%%%%%%%
		&=\frac{1}{2}\,\Bigg\{\,\sup_{\tilde x\in B_\eps(x)}u_\eps(\tilde x)-\inf_{\tilde y\in B_\eps(y)}u_{\eps}(\tilde x)\,+\inf_{\tilde x\in B_\eps(x)}u_{\eps}(\tilde y)-\sup_{\tilde y\in B_\eps(y)}u_\eps(\tilde y)\,\Bigg\}\nonumber \\
	&	=\frac{1}{2}\,\Bigg\{\,\sup_{\tilde x\in B_\eps(x), \tilde y\in B_\eps(y)}(u_\eps(\tilde x)-u_{\eps}(\tilde y))\,+\inf_{\tilde x\in B_\eps(x), \tilde y\in B_\eps(y)}(u_{\eps}(\tilde x)-u_\eps(\tilde y))\,\Bigg\}\nonumber \\
				&=\frac{1}{2}\sup_{B_{\eps}(x)\times B_{\eps}(y)}G\,+\,\frac{1}{2}\inf_{B_{\eps}(x)\times B_{\eps}(y)}G\,.
\end{align}
In stochastic terms, what we did above looking at the $\sup$-term on the last and second to last line corresponds to coupling the probability measures so that PI at $x$ wins with probability one simultaneously with PII at $y$, and these simultaneous events form the step of the maximizing player in $\R^{2n}$. 

Similar to (\ref{thesis}), we could consider the $2n$-DPP in (\ref{eq:DPP-2}) at the level of inequality at the maximum point of $u_{\eps}(x)-u_{\eps}(y)-f(x,y)$ to get
\begin{align*}
f(x_0,y_0)&\le \frac{1}{2}\sup_{B_{\eps}(x_0)\times B_{\eps}(y_0)}f\,+\,\frac{1}{2}\inf_{B_{\eps}(x_0)\times B_{\eps}(y_0)}f\,.
\end{align*}
 Then the key property satisfied by the test function $f$ by the explicit computation is the reverse inequality.
\begin{lemma}
Let $x_0, y_0$, $P_{x_0,y_0}$ and $f$ be as above. Then
   $$ \frac{1}{2}\Bigg\{
   \sup_{B_{\eps}(x_0)\times B_{\eps}(y_0)}f\,+\,\inf_{B_{\eps}(x_0)\times B_{\eps}(y_0)}f\Bigg\} -f(x_0,y_0)<0.$$
\end{lemma}

\noindent {\bf Tug-of-war with noise:}
In the case of the tug-of-war with noise analogous argument leads to the monotone $2n$-DPP
\begin{align*}
G(x,y)&=u(x)-u(y)\\
 &=\frac{\a}{2}\left\{ \sup_{B_{\eps}(x)} u + \inf_{B_{\eps}(x)}u \right\}+ \beta
\kint_{B_{\eps}(0)} u (x+h) \ud h\\
&\hspace{1 em}-\frac{\a}{2}\left\{ \sup_{B_{\eps}(y)} u + \inf_{B_{\eps}(y)}u \right\}-\beta
\kint_{B_{\eps}(0)} u (y+h) \ud h\\
&=
\frac{\a}{2}\left\{\sup_{B_{\eps}(x)\times B_{\eps}(y)}G\,+\inf_{B_{\eps}(x)\times B_{\eps}(y)}G\right\}\\
&\hspace{6 em}+\beta\kint_{B_{\eps}(0)} G(x+h,y+P_{x,y}(h)) \ud h.
\end{align*}
However, to obtain a monotone $2n$-DPP as a simple combination of the previous DPPs, we restrict ourselves to the case $p\ge 2$. Similar arguments as above also work. Details can be found in Section 5 of \cite{luirop18}.

\noindent {\bf General coupling formalism and examples:} We could write the couplings in more general terms for suitable given spaces of measures, $\mathcal{A}_1(x)$ and $\mathcal{A}_2(x)$ containing the rules of the games for each player, as
\begin{equation}
\label{DPP}
u(x) =\,\frac{1}{2}\sup_{\mu^{x}_{\I}\in\mathcal{A}_1(x)}\int_{\Rn}u(z)\,d\mu^{x}_{\I}(z)\,
+\frac{1}{2}\inf_{\mu^{x}_{\II}\in\mathcal{A}_2(x)}\int_{\Rn}u(z)\,d\mu^{x}_{\II}(z)\,.
\end{equation} 
Then we can write
\begin{align*}
u(x)-u(y)&=\frac{1}{2}\sup_{\mu^{x}_{\I}\in\mathcal{A}_1(x)}\int_{\Rn}u(z)\,d\mu^{x}_{\I}(z)\,
+\frac{1}{2}\inf_{\mu^{x}_{\II}\in\mathcal{A}_2(x)}\int_{\Rn}u(z)\,d\mu^{x}_{\II}(z)\,
\\
&\hspace{1 em}-\frac{1}{2}\sup_{\mu^{y}_{\I}\in\mathcal{A}_1(y)}\int_{\Rn}u(z)\,d\mu^{y}_{\I}(z)\,
-\frac{1}{2}\inf_{\mu^{y}_{\II}\in\mathcal{A}_2(y)}\int_{\Rn}u(z)\,d\mu^{y}_{\II}(z)\,\\
&=\sup_{\mu^{x}_{\I}\in\mathcal{A}_1(x)}\inf_{\mu^{x}_{\II} \in\mathcal{A}_2(x)}\sup_{\mu^{y}_{\I} \in\mathcal{A}_1(y)}\inf_{\mu^{y}_{\II} \in\mathcal{A}_2(y)}\\
&\hspace{3 em}\bigg(\int_{\Rn}u(z)\,d\left(\frac{\mu^{x}_{\I}+\mu^{x}_{\II}}{2}\right)(z)\,-\int_{\Rn}u(z)\,d\left(\frac{\mu^{y}_{\I}+\mu^{y}_{\II}}{2}\right)(z)\bigg).
\end{align*}
Let $\Gamma\left(\frac{\mu^{x}_{\I}+\mu^{x}_{\II}}{2},\frac{\mu^{y}_{\I}+\mu^{y}_{\II}}{2}\right)$ denote the class of coupled measures whose marginals are $\frac{\mu^{x}_{\I}+\mu^{x}_{\II}}{2}$ and $\frac{\mu^{y}_{\I}+\mu^{y}_{\II}}{2}$. Then for any coupled measure 
\[
\begin{split}
\mu:=\mu_{\half(\mu^{x}_{\I}+\mu^{x}_{\II}),\half(\mu^{y}_{\I}+\mu^{y}_{\II})}\in \Gamma\left(\frac{\mu^{x}_{\I}+\mu^{x}_{\II}}{2},\frac{\mu^{y}_{\I}+\mu^{y}_{\II}}{2}\right), 
\end{split}
\]
 we have
\[
\begin{split}
\int_{\Rn}u(z)\,d\left(\frac{\mu^{x}_{\I}+\mu^{x}_{\II}}{2}\right)(z)\,-&\int_{\Rn}u(z)\,d\left(\frac{\mu^{y}_{\I}+\mu^{y}_{\II}}{2}\right)(z)\\
=&\int_{\R^{2n}}\big(u(z)-u(z')\big)\,d\mu(z,z').
\end{split}
\]
Denoting
\[
\begin{split}
G(x,y):=u(x)-u(y),
\end{split}
\]
we get  the monotone $2n$-DPP 
\begin{equation*}
\label{eq:dpp-2n}
\begin{split}
G(x,y)=&\sup_{\mu^{x}_{\I}\in\mathcal{A}_1(x)}\inf_{\mu^{x}_{\II}\in\mathcal{A}_2(x)}\sup_{\mu^{y}_{\I}\in\mathcal{A}_1(y)}\inf_{\mu^{y}_{\II}\in\mathcal{A}_2(y)} \bigg(\int_{\R^{2n}}G(z,z')\,d\mu(z,z')\,\bigg).
\end{split}
\end{equation*}
So for example in the case of the tug-of-war
\begin{align*}
\frac{\mu^{x}_{\I}+\mu^{x}_{\II}}{2}=\frac{1}{2}\delta^{x}_{\I} + \frac{1}{2}\delta^{x}_{\II}, \quad 
\frac{\mu^{y}_{\I}+\mu^{y}_{\II}}{2}=\frac{1}{2}\delta^{y}_{\I} + \frac{1}{2}\delta^{y}_{\II}.  
\end{align*}
The computation (\ref{eq:DPP-2}) leads to measures
\begin{align*}
\half(\mu_{\mu^{x}_{\I},\mu^{y}_{\II}}+\mu_{\mu^{x}_{\II},\mu^{y}_{\I}}),\quad 
\mu_{\mu^{x}_{\I},\mu^{y}_{\II}}\in \Gamma(\mu^{x}_{\I},\mu^{y}_{\II}),\quad 
\mu_{\mu^{x}_{\II},\mu^{y}_{\I}}\in \Gamma(\mu^{x}_{\II},\mu^{y}_{\I}),
\end{align*}
but since
\begin{align*}
\half(\mu_{\mu^{x}_{\I},\mu^{y}_{\II}}+\mu_{\mu^{x}_{\II},\mu^{y}_{\I}})\in \Gamma\left(\frac{\mu^{x}_{\I}+\mu^{x}_{\II}}{2},\frac{\mu^{y}_{\I}+\mu^{y}_{\II}}{2}\right)
\end{align*}
this is consistent.
Typically, when starting from a game in $\mathbb{R}^n$, several different games in $\mathbb{R}^{2n}$ can be derived. We must then select the game that best suits our purposes.

%%%%%%%%%%
\begin{example}
A further example of coupling is given by the global copying argument used, for example, by Peres and Sheffield in Theorem 1.2(i) of \cite{peress08} or \cite{lewicka20} p.106, see also its analytic counterpart in the proof of Lemma 4.6 \cite{manfredipr12}. The purpose of the argument is to prove global regularity once the regularity is known close to the boundary by ensuring that the game paths started at different points keep their distance until close to the boundary by \lq\lq replicating the path\rq\rq.

First, a player of the game starting at a point can mimic a strategy (i.e.\ the steps) of a player in a game at the other point. This is not a coupling but a choice of strategies and can be justified by the estimate 
\begin{align*}
u_\eps(x)-u_\eps(y)&=\sup_{S_{\I}}\inf_{S_{\II}}\,\mathbb{E}_{S_{\I},S_{\II}}^{x}[F(x_{\tau})]-\sup_{S_{\I}}\inf_{S_{\II}}\,\mathbb{E}_{S_{\I},S_{\II}}^{y}[F(y_{\tau})]\\
&\le \inf_{S_{\II}}\,\mathbb{E}^{x}_{S^{x}_{\I},S_{\II}}[F(x_{\tau})]-\sup_{S_{\I}}\mathbb{E}^{y}_{ S_{\I},S^{y}_{\II}}[F(y_{\tau})]+2\eta\\
&\le  \,\mathbb{E}^{x}_{S^{x}_{\I},\ol S^{x}_{\II}}[F(x_{\tau})]-\mathbb{E}^{y}_{\ol S^{y}_{\I},S^{y}_{\II}}[F(y_{\tau})]+2\eta,
\end{align*}
where $\ol S^{x}_{\II}$ is a translation/copy of $S^y_{\II}$, and $\ol S^{y}_{\I}$ is a translation/copy of $S^{x}_{\I}$. To be precise,  in the first inequality, we should also replace $\tau$ by $\tau^*\le \tau$ so that we stop whenever one of the points gets close enough to the boundary. In fact, such a change in stopping time can be made rigorously as in \cite[Lemma 2.2]{luirops13}, but we omit it here. 

Next we couple the coin tosses so that steps from $\ol S^{x}_{\II}$ and $S^y_{\II}$ as well as from $\ol S^{y}_{\I}$ and $S^{x}_{\I}$ always occur simultaneously. In more formal terms, we use 
\begin{align*}
\half(\mu_{\mu^{x}_{\I},\mu^{y}_{\I}}+\mu_{\mu^{x}_{\II},\mu^{y}_{\II}}),
\quad \mu_{\mu^{x}_{\I},\mu^{y}_{\I}}\in \Gamma(\mu^{x}_{\I},\mu^{y}_{\I}),
\quad \mu_{\mu^{x}_{\II},\mu^{y}_{\II}}\in \Gamma(\mu^{x}_{\II},\mu^{y}_{\II}),
\end{align*}
where each of the measures contains the action of a player and random noise, and we use the above coupling for the parts arising from the players/strategies.
In the case of the random noise part, the coupled measure is indicated by the integral
\begin{align*}
 &\kint_{B_{\eps}(0)}G(x+h, y+P_{x,y}(h)) \ud h,
\end{align*}
where $P_{x,y}(h)=h$ is the identity coupling.

\end{example}

\section{Laplacian and Random Walk}\label{sec:laplacian}

In this section, we work out the connection between the different regularity methods in the context of the Laplacian and the random walk. More precisely, we work out the connection between the method of couplings for the random walk and the Ishii-Lions method for the Laplacian. 

As explained in the previous section, the coupling methods leads $2n$-DPPs with stochastic intuition at the background. 
The plan is first to connect the coupling method to $2n$-dimensional PDEs ($2n$-PDEs) and their maximum principles in Section~\ref{sec:DPP-PDE}.
Later in Section~\ref{sec:pde-ishii-lions}, we work out the connection between the comparison proof for the $2n$-PDE and Ishii-Lions method, thus completing the picture.

\subsection{$2n$-DPP to  $2n$-PDE}
\label{sec:DPP-PDE}
Above in (\ref{eq:2nDPP-p-2}) we derived a $2n$-DPP for random walks
\begin{align}
\label{eq:int-ave}
G(x,y)=\kint_{B_{\eps}(0)}G(x+h,y+P_{x,y}(h))\,dh
\end{align} 
where $P_{x,y}$ is reflection that is 
\begin{align}
\label{eq:nota}
P_{x,y}(h)&:=\left(I-2\frac{x-y}{\abs{x-y}}\otimes \frac{x-y}{\abs{x-y}}\right)h\\
&=:(I-2 \, \eta\otimes \eta)h\nonumber\\
&=:(I-2P)h,\nonumber
\end{align}
where
$x\otimes y$ denotes a matrix with entries $x_iy_j$, and $I$ is an identity matrix.
Next, we derive the corresponding equation. Heuristically, take the Taylor expansion 
\begin{align*}
 G(x+h_x,y+h_y)=&G(x,y)+\left\langle\begin{pmatrix}
D_x G(x,y)\\
D_y G(x,y)
\end{pmatrix},\, \begin{pmatrix}
h_x\\
h_y 
\end{pmatrix}\right \rangle
\\
&+\half \left\langle \begin{pmatrix}
D_{xx}G(x,y)&D_{xy}G(x,y)\\
D_{yx}G(x,y)&D_{yy}G(x,y) 
\end{pmatrix} \begin{pmatrix}
h_x\\
h_y 
\end{pmatrix}, \begin{pmatrix}
h_x\\
h_y 
\end{pmatrix} \right \rangle+\text{error}. 
\end{align*}
At this point $G$ is any smooth function and 
\begin{align*}
&\left\langle \begin{pmatrix}
D_{xx}G&D_{xy}G\\
D_{yx}G&D_{yy}G 
\end{pmatrix} \begin{pmatrix}
h_x\\
h_y 
\end{pmatrix}, \begin{pmatrix}
h_x\\
h_y 
\end{pmatrix} \right\rangle\\
&\hspace{5 em}= \langle D_{xx}G\,  h_x,h_x\rangle +\langle D_{xy}G\,  h_y,h_x\rangle  + \langle D_{yy}G\,  h_y,h_y\rangle\\
&\hspace{5 em}=\sum_{i,j=1}^n G_{x_ix_j}(h_x)_i(h_x)_j+2\, G_{x_iy_j}(h_x)_i(h_y)_j+G_{y_iy_j}(h_y)_i(h_y)_j
\end{align*}
where $G_{x_iy_j}:=G_{x_iy_j}(x,y)$ stands for partial derivatives.
Next,  observe that 
\begin{align*}
\kint_{B_{\eps}(0)}\sum_{i,j=1}^{n}G_{x_ix_j}(h_x)_i(h_x)_j\ud h=\frac{\eps^2}{n+2}\Delta G
\end{align*} 
since the mixed terms vanish by symmetry and the remaining terms can be evaluated using spherical coordinates (this computation is done in detail in \cite[p.13]{parviainen24}).
Similarly using spherical coordinates and  the formula
\begin{align*}
\kint_{B_{\eps}(0)} \langle Ah,Bh \rangle \ud h &=\kint_{B_{\eps}(0)} \langle B^TAh,h \rangle \ud h\\
&=\frac{\eps^2}{n+2} \tr(B^T A)=\frac{\eps^2}{n+2} \tr( A^T B)
\end{align*}
we obtain
\begin{align*}
\kint_{B_{\eps}(0)} &\left \langle D_{xy}G \,h,\left(I-2\frac{x-y}{\abs{x-y}}\otimes \frac{x-y}{\abs{x-y}}\right)h \right\rangle \ud h \\
&\hspace{5 em}=\frac{\eps^2}{n+2} \tr\left\{ D_{xy}G  \left(I-2\frac{x-y}{\abs{x-y}}\otimes \frac{x-y}{\abs{x-y}}\right)\right\}.
\end{align*}
For other terms the computation is similar and we obtain recalling the shorthand notation (\ref{eq:nota}) that
\begin{align}
\label{eq:stoch-motiv}
\kint_{B_{\eps}(0)}&G(x+h,y+P_{x,y}(h))\,dh\\
&=G(x,y)+\frac{\eps^2}{2(n+2)}\tr\left\{ \begin{pmatrix}
I&I-2P\\
I-2P&I 
\end{pmatrix} D^2 G(x,y)\right\}+\text{error},\nonumber
\end{align}
so the coupling method should correspond to the PDE operator above. We have arrived to the corresponding $2n$-PDE
\begin{equation}\label{eq:2npdelaplacian}
\tr\left\{ \begin{pmatrix}
I&I-2P\\
I-2P&I 
\end{pmatrix} D^2 G(x,y)\right\}=0.
\end{equation}

\subsection{2n-PDE and Ishii-Lions method}
\label{sec:pde-ishii-lions}
In this section, we consider harmonic functions
\begin{align*}
   \Delta u=0 \qquad \text{in } B_1. 
\end{align*}
Our first aim, after observing, that 
$u(x)-u(y)$ is a solution to the $2n$-PDE (\ref{eq:2npdelaplacian}),  is to explain how the comparison result is established for such solutions under the smoothness assumption giving 
\begin{align}
\label{eq:comparison-intro}
    u(x)-u(y)\le f(x,y)=C\abs{x-y}^{\delta}.
\end{align}
A similar argument from below would then finish the proof of H\"older continuity.
After this, we explain the connection between this $2n$-PDE proof and the Ishii-Lions method. We also consider the Ishii-Lions method without the smoothness assumptions in the context of viscosity solutions.

To establish (\ref{eq:comparison-intro}), suppose,  seeking  a contradiction,  that 
\begin{align*}
\max_{B_1\times B_1}(u(x)-u(y)-f(x,y))=u(x_0)-u(y_0)-f(x_0,y_0)=:w(x_0,y_0)>0,
\end{align*}
and assume that $(x_0, y_0)$ is an interior local maximum. It turns out that by modifying $f$, we could guarantee such an assumption, but for expository reasons, we postpone this modification to Theorem~\ref{thm:ishii-lions}.
By smoothness, we immediately have at the maximum point that
\begin{align}
\label{eq:max} \nonumber
\begin{pmatrix}
D_{xx}^2u(x_0)&0\\
0&-D_{yy}^2 u(y_0) 
\end{pmatrix}&-\begin{pmatrix}
D_{xx}f(x_0,y_0)&D_{xy}f(x_0,y_0)\\
D_{yx}f(x_0,y_0)&D_{yy}f(x_0,y_0) 
\end{pmatrix}
\\
&=D^{2}w(x_0,y_0)
\le 0, 
\end{align}
giving 
\begin{align*}
\tr\left\{D^2 w(x_0,y_0)\right\}\le  0.
\end{align*}
As from the equation $\Delta  u=0$ we get 
\begin{align*}
\tr\left\{ \begin{pmatrix}
D_{xx}^2u(x_0)&0\\
0&-D_{yy}^2 u(y_0) 
\end{pmatrix}\right\}=0
\end{align*}
one might hope that also 
\begin{align}
\label{eq:wrong}
\tr\left\{\begin{pmatrix}
D_{xx}f(x_0,y_0)&D_{xy}f(x_0,y_0)\\
D_{yx}f(x_0,y_0)&D_{yy}f(x_0,y_0) 
\end{pmatrix}\right\}
\end{align}
would be strictly negative and this would give
\begin{align*}
\tr\left\{D^2 w(x_0,y_0)\right\}>0,
\end{align*}
leading to a contradiction. So we would not have a positive maximum point. However, (\ref{eq:wrong}) is not strictly negative as verified in (\ref{eq:wrong-comp}) below. 

Instead,  we  argue
as follows. 
Recall $G(x,y)=u(x)-u(y)$ where $u$ is a smooth solution to $\Delta u=0$. From the expression
\begin{align*}
D^2 G(x,y)=\begin{pmatrix}
D^2_{xx}u(x)&0\\
0& -D^{2}_{yy}u(y)
\end{pmatrix}
\end{align*}
we deduce
\begin{align}
\label{eq:2n-eq}
\tr\left\{ \begin{pmatrix}
I&B\\
B^T&I 
\end{pmatrix} D^2 G(x,y)\right\}=0
\end{align}
for an arbitrary matrix $B$ that we will choose satisfying the additional condition   
\begin{align}
\label{eq:additional}
    \begin{pmatrix}
I&B\\
B^T&I 
\end{pmatrix}\ge 0.
\end{align}
This additional condition is required to utilize \eqref{eq:max} i.e.\ $D^2 w(x_0,y_0)=D^2G(x_0,y_0)-D^2f(x_0,y_0)\le 0$ so that
\begin{align*}
 \tr\left\{ \begin{pmatrix}
I&B\\
B^T&I 
\end{pmatrix} D^2 w(x_0,y_0)\right\}\le 0,
\end{align*}
which can be seen for example using the spectral decomposition for the coefficient matrix.
Combining this with (\ref{eq:2n-eq})
we obtain
\begin{align}
0\le
\tr\left\{ \begin{pmatrix}
I&B\\
B^T&I 
\end{pmatrix} D^2 f(x_0,y_0)\right\}.
\end{align}
If we can find a matrix $B$ such that (\ref{eq:additional}) holds and 
\begin{align}
\label{eq:key}
\tr\left\{ \begin{pmatrix}
I&B\\
B^T&I 
\end{pmatrix} D^2 f(x_0,y_0)\right\}<0
\end{align}
we would get a contradiction. We need the following
\begin{align}
\label{eq:DD}
Df(x_0,y_0)&=C\delta\abs{x_0-y_0}^{\delta-2}(x_0-y_0,-(x_0-y_0)),\nonumber\\
 D^2f(x_0,y_0)&=\begin{pmatrix}
M&-M\\
-M&M 
\end{pmatrix},
\end{align}
with 
\begin{align*}
M&=C\delta\abs{x_0-y_0}^{\delta-2}\Big((\delta-2)\frac{x_0-y_0}{\abs{x_0-y_0}}\otimes \frac{x_0-y_0}{\abs{x_0-y_0}}+I\Big)
\\
&=C\delta\abs{x_0-y_0}^{\delta-2}\Big((\delta-2)\eta\otimes\eta+I\Big)\\
&=C\delta\abs{x_0-y_0}^{\delta-2}\Big((\delta-2)P+I\Big),
\end{align*}
recalling the shorthands $P =\eta\otimes\eta$,
where $\eta=\frac{x_0-y_0}{|x_0-x_0|}$.
Now, motivated by (\ref{eq:stoch-motiv}),  we use the coefficient matrix there and obtain
\begin{align}
\label{eq:key-comp-2}
\tr&\left\{ \begin{pmatrix}
I&I-2P\\
I-2P&I 
\end{pmatrix} D^2 f(x_0,y_0)\right\}\\
&=\tr\left\{ \begin{pmatrix}
I&I-2P\\
I-2P&I 
\end{pmatrix} \begin{pmatrix}
M&-M\\
-M&M 
\end{pmatrix} \right\}\nonumber\\
&=4 \tr\{PM\}\nonumber\\
&=C\delta\abs{x_0-y_0}^{\delta-2}4 \tr\{(\eta\otimes \eta)\big((\delta-2)\eta\otimes \eta+I\big)\}\nonumber\\
&=4C\delta\abs{x_0-y_0}^{\delta-2}(\delta-1)<0,\nonumber
\end{align}
a contradiction. 
Also, observe that the coefficient matrix used above is positive semidefinite, as shown after (\ref{eq:eigenvalues}). If instead, we would try to use (\ref{eq:wrong}), we would get
\begin{align}
\label{eq:wrong-comp}
    \tr\left\{D^2 f(x_0,y_0)\right\}=2\tr\{M\}=2C\delta\abs{x_0-y_0}^{\delta-2}(\delta-2+n)
\end{align}
and no contradiction would arise.

Next, we recall the Ishii-Lions method for viscosity solutions \cite{ishiil90}. The modification of the comparison function is also presented.  
After the proof, we explain how this connects with the comparison principle for $2n$-PDEs that was discussed above.
%%%%%%%%%%%%%%%%%%%%%%%%%%%
\begin{theorem}[Ishii-Lions]
\label{thm:ishii-lions}
Let $u\in C(\ol B_1(0))$ be a continuous viscosity solution to $\Delta u=0$.
Then for $\delta\in (0,1)$ there is $C>0$ such that
\[
\begin{split}
\abs{u(x)-u(y)}\le C\abs{x-y}^{\delta} \text{ for all } x,y\in B_{1/4}(0).
\end{split}
\]

\end{theorem}
%%%%%%%%%%%%%%%%%%%%%%%%
%%%%%%%%%%%%%%%%%%%%%%%%
\begin{proof}
By considering $(u-\inf_{B_1} u)/(\sup u_{B_1}-\inf_{B_1} u)$, we may assume that
\[
\begin{split}
0\le u \le 1.
\end{split}
\]
Choose $z_0\in B_{1/4}$ and set 
 $$
 \ol f(x,y):=f(x,y)+2\abs{x-z_0}^2:=C\abs{x-y}^\delta+2\abs{x-z_0}^2,
 $$
 where the purpose of the second term is to guarantee the interior maximum below.
  We want to show that $u(x)-u(y)-\ol f(x,y)\le 0$, and thriving for a contradiction  suppose that there is $\theta>0$ and $x_0,y_0\in B_1$ such that
 \[
\begin{split}
u(x_0)-u(y_0)-\ol f(x_0,y_0)=\sup_{(x,y)\in \ol B_1\times \ol B_1}(u(x)-u(y)-\ol f(x,y))&= 
 \theta>0.
\end{split} 
\]
We immediately observe that $x_0 \neq z_0$ by the counter assumption.
Also $(x_0,z_0)\notin \partial (B_1\times B_1)$ since if $(x_0,z_0)\in \partial (B_1\times B_1)$
\[
\begin{split}
C\abs{x_0-y_0}^\delta+2\abs{x_0-z_0}^2\ge 1,
\end{split}
\]
for $C$ large enough.

  Thus we may use the theorem on sums (see for example \cite[Lemma 3.6]{koike04}, or \cite[Theorem 3.3.2]{giga06}) to obtain symmetric matrices $X$ and $Y$ such that
\[
\begin{split}
(D_xf(x_0,y_0),X)&\in \ol J^{2,+}(u(x_0)-2\abs{x_0-z_0}^2),\\
(-D_yf(x_0,y_0),Y)&\in \ol J^{2,-}u(y_0)
\end{split}
\]
i.e.\ by \cite[Proposition 2.7]{koike04}
\[
\begin{split}
(D_xf(x_0,y_0)+4(x_0-z_0),X+4I)&\in \ol J^{2,+}u(x_0),\\
(-D_yf(x_0,y_0),Y)&\in \ol J^{2,-}u(y_0),
\end{split}
\]
with the estimate
\begin{equation}
\label{eq:thm-sum-est}
\begin{split}
\begin{pmatrix}
X&0\\
0&-Y 
\end{pmatrix}\le D^2 f(x_0,y_0)+\frac{1}{\mu} (D^2 f(x_0,y_0))^2. 
\end{split}
\end{equation}
Recalling (\ref{eq:DD}), we also observe
\[
\begin{split}
(D^2f(x_0,z_0))^2=2\begin{pmatrix}
M^2&-M^2\\
-M^2&M^2 
\end{pmatrix}.
\end{split}
\]
where 
\[
\begin{split}
&M^2=C^2\delta^2\abs{x_0-y_0}^{2(\delta-2)}
\\&\hspace{2 em}\cdot\Big((\delta-2)^2\frac{x_0-y_0}{\abs{x_0-y_0}}\otimes \frac{x_0-y_0}{\abs{x_0-y_0}}+2(\delta-2)\frac{x_0-z_0}{\abs{x_0-y_0}}\otimes \frac{x_0-y_0}{\abs{x_0-y_0}}+I\Big)\\
&=C^2\delta^2\abs{x_0-y_0}^{2(\delta-2)} \Big(\underbrace{(\delta^2-4\delta+4+2\delta-4)}_{=\delta(\delta-2)}\frac{x_0-y_0}{\abs{x_0-y_0}}\otimes \frac{x_0-y_0}{\abs{x_0-y_0}}+I\Big).
\end{split}
\]

%%%%%%%%%%%%%%%%%
%%%%%%%%%
First, all the eigenvalues of $X-Y$ are nonpositive since the right hand side of (\ref{eq:thm-sum-est}) annihilates vectors of the form 
$
\begin{pmatrix}
\xi,
\xi
\end{pmatrix}^T
$.
Moreover  by using 
$\begin{pmatrix}
\xi,
-\xi 
\end{pmatrix}^T$
in (\ref{eq:thm-sum-est}) we obtain
\begin{equation}
\label{eq:sep-ests}
\begin{split}
\langle X\xi, \xi\rangle -\langle Y\xi, \xi\rangle&\le 4 \left\langle \left(M+\frac2{\mu}M^2\right)\xi ,  \xi\right\rangle.
\end{split}
\end{equation}
Next we observe that
\begin{align}
\label{eq:ishii-lions-computation}
& \, 4\left\langle (M+\frac{2}{\mu} M^2 )\frac{x_0-y_0}{\abs{x_0-y_0}},\frac{x_0-y_0}{\abs{x_0-y_0}} \right \rangle \\
= & \,
4C\delta\abs{x_0-y_0}^{\delta-2}(\delta-2+1)+\frac{8}{\mu} C^2\a^2\abs{x_0-y_0}^{2(\delta-2)}(\a(\delta-2)+1)\nonumber\\
\le & \,  3C\delta\abs{x_0-y_0}^{\delta-2}(\delta-1)    \nonumber
\end{align}
by choosing $\mu$ such that 
\[
\begin{split}
\frac{8}{\mu} C^2\a^2\abs{x_0-y_0}^{2(\delta-2)}(\a(\delta-2)+1)<- C\delta\abs{x_0-y_0}^{\delta-2}(\delta-2+1).
\end{split}
\] 
This together with (\ref{eq:sep-ests}) implies that one of the eigenvalues of $X-Y$ will have to be smaller than $3C\delta\abs{x_0-y_0}^{\delta-2}(\delta-1)$. This together with nonpositivity of eigenvalues implies for large enough $C>0$ that
\begin{equation}\label{eq:traces}
\begin{split}
0\le \tr(X+4I)-\tr(Y)&\le 4n+ 3C\delta\abs{x_0-y_0}^{\delta-2}(\delta-1)\\
&\le 4n+ 3C\delta 2^{\delta-2}(\delta-1)<0,
\end{split}
\end{equation}
a contradiction.
To obtain the strict inequality above  we used the fact that $\abs{x_0-y_0}\le 2$  and $\delta-1<0$. 

This completes the proof of 
\begin{align*}
u(x)-u(y)\le C\abs{x-y}^\delta+2\abs{x-z_0}^2.
\end{align*}
Suppose then that we want to obtain $\abs{u(\ol x)-u(\ol y)}\le C\abs{\ol x-\ol y}^\delta$ for some given $\ol x,\ol y\in B_{1/4}$. We may always choose the points so that $u(\ol x)\ge u(\ol y)$, and we may select $z_0=\ol x$. Thus the claim follows from the previous estimate.
\end{proof}

%%%%%%%%%%%%%%%%%%%%%%%%%%

Now we explain how this formulation of the Ishii-Lions method connects with the comparison principle for $2n$-PDEs. 
Recall that $I-2P$ is the reflection with respect to the orthogonal complement of $\eta$. Now take any of the $n-1$ orthonormal basis vectors of the orthogonal complement and denote it by $\nu$. Observe that since the reflection leaves these vectors untouched, we get
\begin{align}
\label{eq:eigenvalues}
\begin{pmatrix}
I&I-2P\\
I-2P&I 
\end{pmatrix} \begin{pmatrix}
\nu\\
\nu
\end{pmatrix}=2 \begin{pmatrix}
\nu\\
\nu
\end{pmatrix}.  
\end{align}
We have found $n-1$ eigenvectors with the eigenvalue 2. Then using that the reflections flips $\eta$, we get
 \begin{align*}
\begin{pmatrix}
I&I-2P\\
I-2P&I 
\end{pmatrix} \begin{pmatrix}
\eta\\
-\eta
\end{pmatrix}=2 \begin{pmatrix}
\eta\\
-\eta
\end{pmatrix}.  
\end{align*}
Similarly, we see that 
\begin{align*}
\begin{pmatrix}
I&I-2P\\
I-2P&I 
\end{pmatrix} \begin{pmatrix}
\eta\\
\eta
\end{pmatrix}=0= 
\begin{pmatrix}
I&I-2P\\
I-2P&I 
\end{pmatrix} \begin{pmatrix}
\nu\\
-\nu
\end{pmatrix}.
\end{align*}
So we have found $2n$ eigenvectors, and see that the matrix is positive (semidefinite) as stated in the previous section. 

Now, in the proof of Theorem~\ref{thm:ishii-lions}, the key point is to use theorem on sums. Assuming for expository reasons that everything is smooth and we drop the localization term in Ishii-Lions so that (\ref{eq:thm-sum-est}) reads as
\begin{align}
\label{eq:thm-sum-est-simplified}
\begin{pmatrix}
D_{xx}u&0\\
0&-D_{yy}u 
\end{pmatrix}\le D^2 f(x,y),
\end{align}
with
\begin{align*}
 D^2f(x,y)&=\begin{pmatrix}
M&-M\\
-M&M 
\end{pmatrix},
\end{align*}
and
\begin{align*}
M=C\delta\abs{x-y}^{\delta-2}\Big((\delta-2)\frac{x-y}{\abs{x-y}}\otimes \frac{x-y}{\abs{x-y}}+I\Big).
\end{align*}
Then estimation starting from (\ref{eq:sep-ests}) can be written as 
\begin{align*}
\begin{pmatrix}
\eta\\
-\eta
\end{pmatrix}^T
\begin{pmatrix}
D_{xx}u&0\\
0&-D_{yy}u 
\end{pmatrix} 
\begin{pmatrix}
\eta\\
-\eta
\end{pmatrix}
&\le 
\begin{pmatrix}
\eta\\
-\eta
\end{pmatrix}^T
\begin{pmatrix}
M&-M\\
-M&M 
\end{pmatrix}
\begin{pmatrix}
\eta\\
-\eta
\end{pmatrix}
\\
& =4\langle M\eta, \eta\rangle=  4C\delta\abs{x-y}^{\delta-2}(\delta-1)<0.
\end{align*}
 Then the proof of Theorem~\ref{thm:ishii-lions} continues to produce a contradiction by saying (taking into account our simplifications here) that together with nonpositivity of eigenvalues  this implies
 \begin{align*}
  0\le \tr(X)-\tr(Y)&\le 4C\delta\abs{x-y}^{\delta-2}(\delta-1)<0.
\end{align*}
If we write the whole argument directly using notation in (\ref{eq:thm-sum-est-simplified}), it reads as
\begin{align*}
\tr\{D_{xx}u-D_{yy}u\}&=\sum_{i=1}^n\xi_i^T
\begin{pmatrix}
D_{xx}u&0\\
0&-D_{yy}u 
\end{pmatrix} 
\xi_i
\le 
 \sum_{i=1}^n
 \xi_i^T
\begin{pmatrix}
M&-M\\
-M&M 
\end{pmatrix}
\xi_i\\
&= \sum_{i=1}^n
 \xi_i^T
D^2 f
\xi_i \le 4C\delta\abs{x-y}^{\delta-2}(\delta-1)<0,
\end{align*}
where $\xi_i$ are the eigenvectors found around (\ref{eq:eigenvalues}) related to the eigenvalues $2$, and the first inequality follows from the fact that the components of the vectors there form an orthonormal basis. But now we can write using the linearity of the trace and the fact that the eigenvectors form a spectral decomposition that
\begin{align*}
\sum_{i=1}^n\xi_i^T
D^2 f
\xi_i=\sum_{i=1}^n \tr\left\{
D^2 f 
\xi_i\otimes\xi_i\right\}&=
 \tr\left\{
D^2 f
\sum_{i=1}^n\xi_i\otimes\xi_i\right\}\\
&=
\half  \tr\left\{
D^2 f
\begin{pmatrix}
I&I-2P\\
I-2P&I 
\end{pmatrix} 
\right\}.
\end{align*}
This gives a connection to the $2n$-PDE.

Finally, we could have presented the Ishii-Lions method and the theorem on sums in the form, where (\ref{eq:thm-sum-est}) is multiplied by the coeffiecient matrix to get 
\begin{align*}
 \tr&\left\{ \begin{pmatrix}
I&I-2P\\
I-2P&I 
\end{pmatrix}
\begin{pmatrix}
X&0\\
0&-Y 
\end{pmatrix}\right\}\\
&\le  \tr\left\{\begin{pmatrix}
I&I-2P\\
I-2P&I 
\end{pmatrix} \Big( D^2 f(x_0,y_0)+\frac{1}{\mu} (D^2 f(x_0,y_0))^2\Big)
\right\},
\end{align*}
see \cite{ishiil90,porrettap13}. 
This would have given a more direct access to the proof of the principle of comparison for $2n$ -PDEs, but we decided to first follow the presentation in \cite[Theorem 6.1]{parviainen24}, where the estimates are produced from the theorem on sums multiplied by vectors from left and right. This might be a slightly simpler way to use the theorem, and this approach is quite common in the literature of viscosity solutions in any case. Moreover, the approaches are equivalent as explained above, and we present the second approach when dealing with the $p$-Laplacian in Theorem \ref{thm: ishii-lions p-laplace}.

\section{Infinity Laplacian and Tug-of-war}
\label{sec:infty}
\subsection{2n-DPP to $2n$-PDE}

We again start with the $2n$-DPP method arising from the tug-of-war game, and derive the corresponding $2n$-dimensional partial differential equation. After this we work out the comparison proof for the $2n$-PDE. The Ishii-Lions method is postponed to Section \ref{sec:p-lap}, where we cover the whole range of $p$ in the context of the $p$-Laplacian.

As explained in Section~\ref{sec:intuition}, we need to look at the DPP with doubled variables
\begin{align*}
G(x,y)=\frac{1}{2}\sup_{B_{\eps}(x)\times B_{\eps}(y)}G\,+\,\frac{1}{2}\inf_{B_{\eps}(x)\times B_{\eps}(y)}G\,.
\end{align*}
Let us make a formal Taylor approximation for any smooth $G$ defined in $\R^{2n}$ with non-vanishing gradient. We set  
\begin{align*}
h_x=\eps\frac{D_x G}{\abs{D_x G}}:=\eps\frac{D_x G(x,y)}{\abs{D_x G(x,y)}},\quad h_y=\eps\frac{D_y G}{\abs{D_y G}}:=\eps\frac{D_y G(x,y)}{\abs{D_y G(x,y)}}
\end{align*}
to obtain
\begin{align*}
\sup_{B_{\eps}(x)\times B_{\eps}(y)}G&\approx
 G(x+h_x,y+h_y)=G(x,y)+\left\langle \begin{pmatrix}
D_x G(x,y)\\
D_y G(x,y)
\end{pmatrix}, \begin{pmatrix}
h_x\\
h_y 
\end{pmatrix}\right\rangle
\\
&+\half \left\langle \begin{pmatrix}
D_{xx}G(x,y)&D_{xy}G(x,y)\\
D_{yx}G(x,y)&D_{yy}G(x,y) 
\end{pmatrix} \begin{pmatrix}
h_x\\
h_y 
\end{pmatrix}, \begin{pmatrix}
h_x\\
h_y 
\end{pmatrix} \right\rangle+\text{error}. 
\end{align*}
Selecting 
\begin{align*}
h_x=-\eps\frac{D_x G}{\abs{D_x G}},\quad h_y=-\eps\frac{D_y G}{\abs{D_y G}}
\end{align*}
we get an approximation for $\inf_{B_{\eps}(x)\times B_{\eps}(y)}G$. Then summing the approximation for the $\inf$ and $\sup$, we see that the first order terms cancel out and we get
\begin{align}
\label{eq:taylor-sup-inf}
\frac{1}{2}&\sup_{B_{\eps}(x)\times B_{\eps}(y)}G\,+\,\frac{1}{2}\inf_{B_{\eps}(x)\times B_{\eps}(y)}G
\\&\approx \frac{\eps^{2}}{2} \left\langle \begin{pmatrix}
D_{xx}G&D_{xy}G\\
D_{yx}G&D_{yy}G 
\end{pmatrix} \begin{pmatrix}
\frac{D_x G}{\abs{D_x G}}\\
\frac{D_y G}{\abs{D_y G}}
\end{pmatrix}, \begin{pmatrix}
\frac{D_x G}{\abs{D_x G}}\\
\frac{D_y G}{\abs{D_y G}}
\end{pmatrix} \right\rangle+\text{error}\nonumber
\\
&=\frac{\eps^{2}}{2} \tr\left\{ \begin{pmatrix}
\frac{D_x G}{\abs{D_x G}}\\
\frac{D_y G}{\abs{D_y G}}
\end{pmatrix}\otimes \begin{pmatrix}
\frac{D_x G}{\abs{D_x G}}\\
\frac{D_y G}{\abs{D_y G}}
\end{pmatrix} \begin{pmatrix}
D_{xx}G&D_{xy}G\\
D_{yx}G&D_{yy}G 
\end{pmatrix} \right\}+\text{error}\nonumber
\end{align}
where we used the shorthand notation $D_{xx}G:=D_{xx}G(x,y)$ etc.
We formally obtain the equation  
\begin{equation}\label{eq:2npdeinfinity}
    \tr\left\{ \begin{pmatrix}
\frac{D_x G}{\abs{D_x G}}\\
\frac{D_y G}{\abs{D_y G}}
\end{pmatrix}\otimes \begin{pmatrix}
\frac{D_x G}{\abs{D_x G}}\\
\frac{D_y G}{\abs{D_y G}}
\end{pmatrix} \begin{pmatrix}
D_{xx}G&D_{xy}G\\
D_{yx}G&D_{yy}G 
\end{pmatrix} \right\}=0. \end{equation}
Note that this equation is not $\Delta^N_\infty G(x,y)=0
$
since for $\Delta^N_\infty G(x,y)$ the vector  $(
\frac{D_x G}{\abs{D_x G}},
\frac{D_y G}{\abs{D_y G}}
)
$
is not equal to
$$\left(
\frac{D_x G}{\sqrt{\abs{D_x G}^2+\abs{D_y G}^2}},
\frac{D_y G}{{\sqrt{\abs{D_x G}^2+\abs{D_y G}^2}}}
\right).
$$

\subsection{2n-PDE and Ishii-Lions method}\label{subsec:2npdeInfty}
Next we consider the $2n$-PDE approach.
Let $u$ be a  smooth infinity harmonic function with non-vanishing gradient and write
\begin{align*}
\Delta_{\infty}^N u = \left\langle D^2u\frac{D u}{\abs{D u}} ,\frac{D u}{\abs{D u}} \right\rangle =\tr\left\{\left(\frac{Du}{|Du|}\otimes \frac{Du}{|Du|}\right) D^2u \right\}=0.
\end{align*}
We want to prove
 \begin{align*}
         u(x)-u(y)\le f(x,y)=C\abs{x-y}^{\delta}.
 \end{align*}
To this end, thriving for a contradiction, assume that
\begin{align*}
\max_{B_1\times B_1}(u(x)-u(y)-f(x,y))=u(x_0)-u(y_0)-f(x_0,y_0):=w(x_0,y_0)>0,
\end{align*}
is attained at  an interior point. Again, this assumption could be guaranteed by a modification of $f$, but for expository reasons we omit it at this point. This immediately implies $x_0\not= y_0$.
We set $G(x,y)=u(x)-u(y)$ as before.

Let us again consider the maximum point as in (\ref{eq:max}) so that $Dw(x,y)=0$ i.e.\ $Du(x_0)=D_xG(x_0,y_0)=D_xf(x_0,y_0),\ 
-Du(y_0)=D_y G(x_0,y_0)=D_y f(x_0,y_0)$ and $D^2w(x_0,y_0)\le 0$. Let us write 
\begin{align*}
    \eta&=\frac{D_xG(x_0,y_0)}{|D_xG(x_0,y_0)|}=\frac{Du(x_0)}{|Du(x_0)|}\\
    &=\frac{D_x f(x_0,y_0)}{|D_xf(x_0,y_0)|}=\frac{x_0-y_0}{\abs{x_0-y_0}}= -\frac{D_y f(x_0,y_0)}{|D_yf(x_0,y_0)|}\\
    &=-\frac{D_yG(x_0,y_0)}{|D_yG(x_0,y_0)|}
    =\frac{Du(y_0)}{|Du(y_0)|}.
\end{align*}
Since $u$ is a solution we can use (\ref{eq:2npdeinfinity}) and write
\begin{align*}
0&=\tr\left\{\begin{pmatrix}
\eta\\
-\eta
\end{pmatrix}\otimes \begin{pmatrix}
\eta\\
-\eta
\end{pmatrix}\begin{pmatrix}
    D^2u(x_0)&  0 \\
    0 &  -D^2u(y_0)
\end{pmatrix}\right\}\\
&=\tr\left\{\begin{pmatrix}
\eta\otimes\eta& -\eta\otimes \eta\\
-\eta\otimes\eta &  \eta\otimes\eta
\end{pmatrix}\begin{pmatrix}
    D^2u(x_0)&  0 \\
    0 &  -D^2u(y_0)
\end{pmatrix}\right\}.    
\end{align*}
Using this and the fact that $D^2 w(x_0,y_0)\le 0$, we conclude, since the coefficient matrix  
is positive semidefinite, that
\begin{align}
\label{eq:matrixproduct1}
0&\ge \tr \left\{\begin{pmatrix}
\eta\otimes\eta& -\eta\otimes \eta\nonumber \\
-\eta\otimes \eta &  \eta\otimes\eta\end{pmatrix}
D^2 w(x_0,y_0)\right\}\\
&=-\tr \left\{\begin{pmatrix}
\eta\otimes\eta& -\eta\otimes \eta\\
-\eta\otimes \eta &  \eta\otimes\eta\end{pmatrix}
D^2 f(x_0,y_0)\right\}.    
\end{align}

By this and \eqref{eq:DD}, we get
\begin{align}
\label{eq:key-comb-infty}
0&\le 
\tr\left\{ \begin{pmatrix}
\eta\otimes\eta& -\eta\otimes \eta\\
-\eta\otimes \eta &  \eta\otimes\eta
\end{pmatrix} D^2 f(x_0,y_0)\right\} \\
&=
\left\langle \begin{pmatrix}
M&-M\\
-M&M 
\end{pmatrix}  \begin{pmatrix}
\eta\\
-\eta
\end{pmatrix}, 
\begin{pmatrix}
\eta\\
-\eta
\end{pmatrix}
\right\rangle 
\nonumber\\
&=4\langle M\eta,\eta \rangle \nonumber\\
&=C\delta\abs{x_0-y_0}^{\delta-2}4 \tr\{\eta\otimes \eta\big((\delta-2)\eta\otimes \eta+I\big)\}\nonumber\\
&=4C\delta\abs{x_0-y_0}^{\delta-2}(\delta-1)<0,\nonumber
\end{align}
a contradiction. Observe that the negative number obtained above is exactly the same as in the linear case.
A rigorous proof would use viscosity solutions and localization terms as in Section \ref{sec:laplacian} above. Also observe that it would have been enough to use the equation 
\begin{align}
\label{eq:simple-eq}
     0=\tr\left\{\begin{pmatrix}
\eta\otimes\eta& 0\\
0&  \eta\otimes\eta
\end{pmatrix}\begin{pmatrix}
    D^2u(x_0)&  0 \\
    0 &  -D^2u(y_0)
\end{pmatrix}\right\}
\end{align}
above, and the only change would have been that we get $2$ instead $4$ in the previous estimate.

We postpone the Ishii-Lions method to Theorem~\ref{thm:ishii-lions-infty}.

\section{ $p$-Laplacian and Tug-of-War with Noise}
\label{sec:p-lap}
In this section, we work out the connection between the different regularity methods in the context of the $p$-Laplacian and the tug-of-war with noise.

\subsection{2n-DPP to $2n$-PDE } 
The $2n$-DPP is obtained by combining the cases $p=2$ and $p=\infty$ as pointed out in Section~\ref{sec:intuition} whenever $p\ge 2$,
\begin{align*}
G(x,y)&=\frac{\a}{2}\left\{\sup_{B_{\eps}(x)\times B_{\eps}(y)}G\,+\inf_{B_{\eps}(x)\times B_{\eps}(y)}G\right\}\\
&\hspace{10 em}+\beta\kint_{B_{\eps}(0)} G(x+h,y+P_{x,y}(h)) \ud h.
\end{align*}
The $2n$-equation is just a linear combination of \eqref{eq:2npdelaplacian} and (\ref{eq:2npdeinfinity})
\begin{align}
 \label{eq:2npdeplaplacian} \nonumber
0=&(p-2)\tr\left\{ \begin{pmatrix}
\frac{D_x G}{\abs{D_x G}}\\
\frac{D_y G}{\abs{D_y G}}
\end{pmatrix}\otimes 
\begin{pmatrix}
\frac{D_x G}{\abs{D_x G}}\\
\frac{D_y G}{\abs{D_y G}}
\end{pmatrix} D^2G
\right\}\\
&+\tr\left\{ \begin{pmatrix}
I&I-2P\\
I-2P&I 
\end{pmatrix} D^2 G\right\}.
\end{align}

Recall that $P$ is the $n\times n$ matrix $\frac{x-y}{|x-y|}\otimes \frac{x-y}{|x-y|}$.
\subsection{2n-PDE and Ishii-Lions method}
For $1<p<\infty$, we consider a smooth solution $u$ with non-vanishing gradient of  the normalized $p$-Laplacian
\begin{align*}
\Delta_{p}^N u &=\Delta u + (p-2)\left\langle D^2u\frac{D u}{|Du|}
,\frac{D u}{\abs{D u}} \right\rangle \\
&=\tr\left\{\left(I+(p-2)\left(\frac{Du}{|Du|}\otimes \frac{Du}{|Du|}\right)\right) D^2u \right\}=0.
\end{align*}

Similarly as before we want to prove that
 \begin{align*}
         u(x)-u(y)\le f(x,y)=C\abs{x-y}^{\delta},
 \end{align*}
and assume that
\begin{align*}
\max_{\ol B_1\times \ol B_1}(u(x)-u(y)-f(x,y))=u(x_0)-u(y_0)-f(x_0,y_0):=w(x_0,y_0)>0
\end{align*}
is attained at an interior point with $x_0 \neq y_0$ similarly as in the case $p=\infty$. We again, denote the normalized gradient term
\begin{align*}
    \eta:&=\frac{D_xG(x_0,y_0)}{|D_xG(x_0,y_0)|}=\frac{Du(x_0)}{|Du(x_0)|}\\
    &=\frac{D_x f(x_0,y_0)}{|D_xf(x_0,y_0)|}=\frac{x_0-y_0}{\abs{x_0-y_0}}= -\frac{D_y f(x_0,y_0)}{|D_yf(x_0,y_0)|}\\
    &=-\frac{D_yG(x_0,y_0)}{|D_yG(x_0,y_0)|}
    =\frac{Du(y_0)}{|Du(y_0)|}.
\end{align*}
Now it holds that
\begin{align*}
\begin{pmatrix}
D^2u(x_0)&0\\
0&-D^2u(y_0) 
\end{pmatrix}\le D^2 f(x_0,y_0).
\end{align*}
Next we use the fact that $u$ is a solution and the $2n$-PDE \eqref{eq:2npdeplaplacian}. Notice that for $(x,y) = (x_0,y_0)$ we have $P = \eta \otimes \eta$, where $P$ is the same reflection matrix as in \eqref{eq:2npdeplaplacian}. Thus, the desired contradiction would follow from the computation
\begin{align*}
    0 & = \tr\left\{ \begin{pmatrix}
I + (p-2)\eta \otimes \eta &0\\
0 &I + (p-2)\eta \otimes \eta 
\end{pmatrix} \begin{pmatrix}
D^2 u(x_0) &0\\
0 & D^2 u(y_0) 
\end{pmatrix}
\right\}\\
& = \tr\left\{ \begin{pmatrix}
I + (p-2)\eta \otimes \eta &I - p\eta \otimes \eta\\
I - p\eta \otimes \eta &I + (p-2)\eta \otimes \eta 
\end{pmatrix} \begin{pmatrix}
D^2 u(x_0) & 0 \\
0 & D^2 u(y_0) 
\end{pmatrix}
\right\}\\
& \leq \tr\left\{ \begin{pmatrix}
I + (p-2)\eta \otimes \eta &I - p\eta \otimes \eta\\
I - p\eta \otimes \eta &I + (p-2)\eta \otimes \eta 
\end{pmatrix} D^2f(x_0,y_0)
\right\} < 0.
\end{align*}
Now to justify the first inequality in the previous computation, we need to verify that the matrix
\begin{equation}\label{eq:matrix_p-laplacian}
    \mathcal{A} = \mathcal{A}(\eta,p) := 
    \begin{pmatrix}
        I+(p-2)\eta \otimes \eta  & I -p\,\eta \otimes \eta\\
        I-p\, \eta \otimes \eta &  I+(p-2)\eta \otimes \eta
    \end{pmatrix}
\end{equation}
is positive semi-definite. To see this, fix any orthonormal basis $\{ \xi_1,\dots \xi_{n-1} \}$ of the orthogonal complement of $\eta$. Then we have a linear basis of $\R^{2n}$ consisting of eigenvectors of $\mathcal{A}$:
\begin{align*}
    \mathcal{A} \begin{pmatrix}
        \eta\\
        \eta
        \end{pmatrix}
        & = 0,
        && \mathcal{A} \begin{pmatrix}
        \eta\\
        -\eta
        \end{pmatrix}
         = 2(p - 1)
        \begin{pmatrix}
        \eta\\
        -\eta
        \end{pmatrix},\\
        \mathcal{A} \begin{pmatrix}
        \xi_i\\
        \xi_i
        \end{pmatrix} & = 2\begin{pmatrix}
        \xi_i\\
        \xi_i
        \end{pmatrix},
         && \mathcal{A} \begin{pmatrix}
        \xi_i\\
        -\xi_i
        \end{pmatrix} = 
        0\begin{pmatrix}
        \xi_i\\
        -\xi_i
        \end{pmatrix}.
\end{align*}
Since $p > 1$, $\mathcal{A}$ is positive semi-definite. This would conclude the case when $u$ is smooth.

In \cite[Proposition B.2]{attouchipr17} and \cite[Lemma B.1]{siltakoski21}, the Ishii-Lions method is used to prove the Lipschitz regularity for related equations after deducing H\"older continuity from the Krylov-Safonov theory. Here, the intention is to compare the methods, and therefore we work with the Ishii-Lions approach also in the context of H\"older continuity.

\begin{theorem}[Ishii-Lions]\label{thm: ishii-lions p-laplace} Let $1<p<\infty$ and $u \in C(\ol B_1(0))$ be a viscosity solution to the normalized $p$-Laplace equation $\Delta^N_p u =0$. Then for any $\delta\in (0,1)$ there is $C>0$ such that
\[
\begin{split}
\abs{u(x)-u(y)}\le C\abs{x-y}^{\delta} \text{ for all } x,y\in B_{1/4}(0).
\end{split}
\]
\end{theorem}

\begin{proof}
    Without loss of generality we assume $0 \leq u \leq 1$.
    We let $C > 0$ be a constant determined during the proof, and let $z_0 \in B_{1/4}(0)$ be arbitrary.
    Now consider the function
    \[
        \ol f(x,z) := f(x,z) + 2\abs{x-z_0}^2,
    \]
    where $f(x,z) := C\abs{x-z}^{\delta}$.
    Let $x_0,y_0 \in \overline{B}_1(0)$ so that
    \begin{equation}\label{eq: def theta}
        u(x_0) - u(y_0) - \ol f(x_0,y_0) = \max_{x,y \in \overline{B}_1(0)} u(x) - u(y) - \ol f(x,y) = \theta>0.
    \end{equation}
    The first objective is to prove that $\theta \leq 0$ by deriving a contradiction from the counter assumption $\theta > 0$.

    First note that $x_0,y_0 \notin \partial B_1(0)$. Otherwise, by choosing $C$ to be large enough, we have
    \begin{align*}
        \theta & = u(x_0) - u(y_0) - \ol f(x_0,y_0)\\
        & \leq 1 - f(x_0,y_0)\\
        & = 1 - C\abs{x_0-y_0}^{\delta} -2\abs{x_0-z_0}^2 < 0.
    \end{align*}
    Moreover by the counter assumption, we also have
    \begin{equation}\label{eq: Ishii-Lions x and y are close}
        \abs{x_0-y_0},\abs{x_0 - z_0} \leq 1.
    \end{equation}
    
    Next we fix another constant
    \[
        \mu := C \abs{x_0 - y_0}^{\delta - 2} > 0.
    \]
    By applying theorem on sums in the same way it was used in the proof of Theorem \ref{thm:ishii-lions} (except that now we also recall the lower bound from the references mentioned there), there are symmetric matrices $X$ and $Y$ satisfying
    \begin{align*}
        (D_xf(x_0,y_0)+4(x_0-z_0),X+4I) & \in \overline{J}^{2,+} u(x_0),  \\
        (-D_yf(x_0,y_0),Y) & \in \overline{J}^{2,-} u(y_0)
    \end{align*}
    and the inequalities
    \begin{align}
    \label{eq: Thm of Sums (1)}
    \begin{pmatrix}
    X&0\\
    0&-Y 
    \end{pmatrix} & \leq 
     \begin{pmatrix}
    M&-M\\
    -M&M 
    \end{pmatrix}
    + \frac{2}{\mu}
    \begin{pmatrix}
    M^2&-M^2\\
    -M^2&M^2 
    \end{pmatrix}\\
    \label{eq: Thm of Sums (2)}
     \begin{pmatrix}
    X&0\\
    0&-Y 
    \end{pmatrix} & \geq
    -(\mu + \norm{M})\begin{pmatrix}
    I&0\\
    0&I 
    \end{pmatrix}.
    \end{align}
    Here $M$ is the same matrix as in \eqref{eq:DD}, and $\norm{M}=\max\{\abs{\lambda}\,:\,\lambda \text{ is an eigenvalue of }M\}$. We denote normalized gradient vectors as
\[
    \eta_1 := \frac{D_xf(x_0,y_0)+4(x_0-z_0)}{\abs{D_xf(x_0,y_0)+4(x_0-z_0)}} \text{ and } \eta_2 := \frac{-D_yf(x_0,y_0)}{\abs{D_yf(x_0,y_0)}} = \frac{x_0-y_0}{\abs{x_0-y_0}}.
\]
Recall that $x_0 \neq y_0$ as $\theta > 0$.
By choosing $C$ to be large enough, it follows from \eqref{eq: Ishii-Lions x and y are close} that
\[
    \abs{D_yf(x_0,y_0)} = \abs{D_xf(x_0,y_0)} = C\delta\abs{x_0-y_0}^{\delta - 1} \geq C\delta > 8\abs{x_0-z_0}
\]
and that
\begin{equation}\label{eq: Ishii-Lions Norms of 1st derivatives}
    \abs{D_xf(x_0,y_0)+4(x_0-z_0)}, \abs{D_yf(x_0,y_0)} \geq \frac{C}{2}\delta \abs{x_0-y_0}^{\delta - 1} > 0.
\end{equation}
In particular, $\eta_1,\eta_2$ are well-defined.

Now we use the fact that $u$ is a viscosity solution, and compute
\begin{align}
    0 & \leq \tr\left((I + (p-2)\eta_1 \otimes \eta_1 )(X + 4I)\right) - \tr\left( (I + (p-2)\eta_2 \otimes \eta_2 )Y)\right) \label{eq: Difference (1)}\\
    & = \tr\left( (I + (p-2)\eta_2 \otimes \eta_2) (X - Y)\right) + (p-2)\tr((\eta_1 \otimes \eta_1 - \eta_2 \otimes \eta_2)X)) \nonumber \\
    & + 4\tr(I + (p-2)\eta_1 \otimes\eta_1) \nonumber \\
    & =: T_1 + T_2 + T_3. \nonumber
\end{align}
Then consider the matrix $\mathcal{A} = \mathcal{A}(\eta_2,p)$ from \eqref{eq:matrix_p-laplacian}. By applying the matrix inequality \eqref{eq: Thm of Sums (1)} from theorem on sums and the fact that $\mathcal{A}$ is positive semi-definite, we can estimate $T_1$ by
\begin{align}
    % 1
    T_1 & = \tr\left( (I + (p-2)\eta_2 \otimes \eta_2) (X - Y)\right)\label{eq: Difference (2)}\\
    % 2
    & =  \tr\left\{\mathcal{A}
    \begin{pmatrix}
        X & 0\\
        0 & -Y
    \end{pmatrix}\right\} \nonumber\\
    % 3
    & \leq \tr\left\{\mathcal{A}
    \begin{pmatrix}
        M + \frac{2}{\mu} M^2 & -(M + \frac{2}{\mu}M^2)\\
        -(M + \frac{2}{\mu}M^2) & M + \frac{2}{\mu}M^2
    \end{pmatrix}\right\}\nonumber \\
    % 4
    & = 8(p-1)\left\langle \left(M +\frac{2}{\mu}M^2\right)\eta_2, \eta_2 \right\rangle. \nonumber
    \end{align}
    We proceed by using the computation in \eqref{eq:ishii-lions-computation}, and estimate 
    \begin{align*}
    % 5
    T_1 & \leq 8(p-1)\left\langle \left(M +\frac{2}{\mu}M^2\right)\eta_2, \eta_2 \right\rangle\\
    & =
    8(p-1)C\delta\left(\abs{x_0-y_0}^{\delta-2}(\delta-1)
+\frac{2}{\mu} C\delta\abs{x_0-y_0}^{2(\delta-2)}(\delta(\delta-2)+1)\right)\\
    % 6
& = 8(p-1)C\delta \abs{x_0-y_0}^{\delta - 2} \left( (\delta - 1) + 2\delta(\delta - 1)^2\right)\\
    % 7
& \leq 4(p-1)C\delta \abs{x_0 - y_0}^{\delta - 2}(\delta - 1).
\end{align*}

The last term can easily be computed
\[
    T_3 = 4\tr(I + (p-2)\eta_1 \otimes\eta_1) = 4(n - p-2).
\]

In order to estimate the remaining term $T_2$, we apply the general fact, which is proven in e.g. \cite[Lemma A.1]{siltakoski2022holder}:
\begin{equation*}
    \left| \frac{a}{\abs{a}} - \frac{b}{\abs{b}} \right| \leq \frac{2\abs{a-b}}{\max\{\abs{a}, \abs{b} \} } \quad \text{for all } a,b \neq 0.
\end{equation*}
By applying this and \eqref{eq: Ishii-Lions Norms of 1st derivatives}, we have
\begin{align}
    T_2 & = (p-2)\tr((\eta_1 \otimes \eta_1 - \eta_2 \otimes \eta_2)X) \nonumber \\
    & \leq 2n\abs{p-2}\abs{\eta_1 - \eta_2}\norm{X} \nonumber \\
    & \leq 4n\abs{p-2}\norm{X}\frac{\abs{x_0-z_0}}{\max\{\abs{D_xf(x_0,y_0)+4(x_0-z_0)}, \abs{D_yf(x_0,y_0)}\}} \nonumber \\
    & \leq 8n\abs{p-2}\norm{X}C^{-1}\abs{x_0-y_0}^{1-\delta}. \label{eq:pre-estimateT2}
\end{align}
In the first inequality, we used the estimation
\begin{align*}
\norm{\eta_1 \otimes \eta_1 - \eta_2 \otimes \eta_2}\le  \norm{(\eta_1-\eta_2) \otimes \eta_1 - \eta_2 \otimes (\eta_2-\eta_1)}\le (\abs{\eta_1}+\abs{\eta_2})\abs{\eta_1-\eta_2}
\end{align*}
and in the last step the fact that $\abs{x_0-z_0}\le 1$. Since
\[
    \left\langle
    \begin{pmatrix}
    X&0\\
    0&-Y 
    \end{pmatrix}
    \begin{pmatrix}
    \xi\\
    0 
    \end{pmatrix},\begin{pmatrix}
    \xi\\
    0 
    \end{pmatrix}
    \right \rangle = \langle X\xi,\xi
    \rangle
\]
holds for any unit vector $\xi \in \R^n$,
it follows from the inequalities \eqref{eq: Thm of Sums (1)} and \eqref{eq: Thm of Sums (2)} that
\[
    -(\mu + \norm{M}) \leq \langle X\xi,\xi\rangle \leq \left\langle \left(M + \frac{2}{\mu} M^2 \right)\xi,\xi  \right\rangle,
\]
which further yields
\[
    \norm{X} \leq \mu + \norm{M} + \frac{2}{\mu} \norm{M^2}.
\]
Using this, the precise expressions of $M$ and $M^2$ given in \eqref{eq:DD} and in the proof of Theorem \ref{thm:ishii-lions} respectively, and by recalling the definition of $\mu$, we have the estimate
\begin{align*}
    \norm{X} \leq C\abs{x_0 - y_0}^{\delta - 2}\left( 1 + \delta(1-\delta) + 2\delta(\delta(2-\delta) + 1) \right).
\end{align*}
Combining this with \eqref{eq:pre-estimateT2}, we end up with
\[
    T_2 \leq L \abs{x_0 - y_0}^{-1}
\]
for some $L > 0$ independent of $C$ and $x_0,y_0$.

Finally, by combining the estimates for $T_1,T_2$ and $T_3$, and by noting that $\delta - 2 < -1$ as well as $\abs{x_0-y_0} \leq 1$, we can choose $C$ to be large enough so that

\begin{align*}
    0 & \leq T_1 + T_2 + T_3\\
    & \leq 4(p-1)C\delta \abs{x_0-y_0}^{\delta-2}(\delta-1) + L \abs{x_0-y_0}^{-1} + 4(n-p-2)\\
    & \leq \abs{x_0-y_0}^{\delta-2}(\delta - 1) < 0.
\end{align*}
Thus, we have arrived to the contradiction.

We are now ready to conclude the desired H\"older regularity of $u$. Let $x,y \in B_{1/4}(0)$ and choose $z_0 := x$. Without loss of generality, we may assume that $u(x) \geq u(y)$.
Since $\theta > 0$ in \eqref{eq: def theta} led into a contradiction, we conclude that
\[
   \abs{u(x) - u(y)} =  u(x) - u(y) \leq f(x,y) = C\abs{x-y}^{\delta}.
\]
\end{proof}

We discuss the analogous result of Theorem \ref{thm: ishii-lions p-laplace} for $p = \infty$.
\begin{theorem}\label{thm: Ishii-Lions Inf}
\label{thm:ishii-lions-infty}
Let $u \in C(\ol B_1(0))$ be a viscosity solution to the normalized infinity Laplace equation $\Delta^N_{\infty} u =0$. Then for any $\delta\in (0,1)$ there is $C>0$ such that
\[
\begin{split}
\abs{u(x)-u(y)}\le C\abs{x-y}^{\delta} \text{ for all } x,y\in B_{1/4}(0).
\end{split}
\]
\end{theorem}

\begin{proof}
    We only highlight the main differences from the proof of Theorem \ref{thm: ishii-lions p-laplace}.
    First, we use a different PDE in the computation \eqref{eq: Difference (1)}.
    Second, during the estimation of $T_1$ in \eqref{eq: Difference (2)}, we replace $\mathcal{A}$ by the positive semidefinite matrix
    \[
        \mathcal{A}' := \begin{pmatrix}
    \eta_2 \otimes \eta_2&0\\
    0 & \eta_2 \otimes \eta_2
    \end{pmatrix}.
    \]
    The rest of the proof is identical.
\end{proof}
%%%%%%%%%%%%%%%%
\begin{remark} 
We have presented a proof for finite $p$ that uses the lower bound in the theorem on sums because it can be easily extended to include the case $p=\infty$. For finite $p$,  one can deduce the lower bound for
$\norm{X}$ from the equation itself.  Recall that
$$       (D_xf(x_0,y_0)+4(x_0-z_0),X+4I)  \in \overline{J}^{2,+} u(x_0),  $$     
so that we have $0  \leq \tr\left((I + (p-2)\eta_1 \otimes \eta_1 )(X + 4I)\right)$. Order the eigenvalues of $X$ as $\lambda_1\le \lambda_2\le \ldots\le \lambda_n$. Note that an upper bound for all $\lambda_i$ follows from \eqref{eq: Thm of Sums (1)}. From the fact that $u$ is a viscosity subsolution we get
$$ 0  \leq \tr\left((I + (p-2)\eta_1 \otimes \eta_1 ) X\right) + 4n +(p-2).$$
From this we have
\begin{align}
    \label{eq:lowerbound} 2-p-4n  &\leq \lambda_1+\ldots+\lambda_n+ (p-2) \langle X \eta_1, \eta_1\rangle\nonumber \\
    &\le  \lambda_1+\ldots+\lambda_{n-1}+ (p-1) \langle X \eta_1, \eta_1\rangle.
\end{align}
Thus we get 
$$2-p-4n \le (n-1+(p-1)) \lambda_n,$$
which provides a lower bound for $\lambda_n$. Since $\lambda_n$ is now bounded, we can use \eqref{eq:lowerbound} to get lower bounds on $\lambda_{n-1}$, $\lambda_{n-2}$, and so on.\par 
\end{remark}

\bibliography{citations2}
\bibliographystyle{alpha}

\end{document}